\documentclass[12pt]{amsart}

\usepackage{amsmath}
\usepackage{amssymb}
\usepackage{latexsym}
\usepackage{graphicx} 
\usepackage{mathrsfs}
\usepackage{mathtools}
\usepackage{amsthm}
\usepackage{tikz}
\usetikzlibrary{cd}
\allowdisplaybreaks[4]

\usepackage[colorlinks,citecolor=blue,linkcolor=blue,linktocpage,unicode]{hyperref}

\numberwithin{equation}{section}
\numberwithin{figure}{section}

\newtheorem{thm}{Theorem}[section]
\newtheorem{prop}[thm]{Proposition}
\newtheorem{lemma}[thm]{Lemma}
\newtheorem{cor}[thm]{Corollary}
\newtheorem{problem}[thm]{Problem}

\theoremstyle{definition}
\newtheorem{example}[thm]{Example}
\newtheorem{defin}[thm]{Definition}
\newtheorem{remark}[thm]{Remark}

\title[Good involutions on symplectic quandles]{On the nonexistence of good involutions on symplectic quandles}
\date{\today}

\author{Yasuhito Nakajima}
\address{Yasuhito Nakajima, Independent Scholar}
\email{yasuhito.nakajima.mathematics@gmail.com}

\author{Kentaro Yamaguchi}
\address{Kentaro Yamaguchi, Research Institute for Mathematical Sciences, Kyoto University, Kyoto 606-8502, Japan}
\email{c201596e@alm.icu.ac.jp}

\begin{document}
\begin{abstract}
We investigate a necessary and sufficient condition for the existence of good involutions on symplectic quandles, which are defined on free $R$-modules equipped with an antisymmetric bilinear form. In particular, we discuss the nonexistence of such good involutions under various algebraic conditions. 
\end{abstract}

\maketitle

\section{Introduction}
\label{sec: introduction}

A \textit{quandle} is an algebraic structure introduced by Joyce \cite{MR638121} and Matveev \cite{MR672410} in knot theory.
From an algebraic viewpoint, a quandle is an idempotent right-invertible self-distributive algebraic structure. This structure appears in many areas of mathematics, for example, as an axiomatization of the properties of conjugation in a group or a generalization of point symmetries of symmetric spaces.

A \textit{symmetric quandle} \cite{MR2657689} is a quandle with a \textit{good involution} \cite{MR2371714}.
Symmetric quandles are useful for studying knots that are not necessarily oriented or orientable.
The existence of good involutions on a given quandle (and hence the existence of a symmetric quandle) is quite nontrivial; thus the following classification problems were posed by Taniguchi \cite{MR4688855}.
\begin{problem}\label{prob: condition}
Find a necessary and sufficient condition for the existence of good involutions on a quandle $X$.
\end{problem}
\begin{problem}\label{prob: isomorphism class}
Determine all isomorphism classes of symmetric quandles associated with a quandle $X$.
\end{problem}

Some classification results are available in the literature.
Kamada--Oshiro \cite{MR2657689} studied the above problems when $X$ is a kei, a trivial quandle, or a dihedral quandle.
Taniguchi \cite{MR4688855} studied the above problems when $X$ is a generalized Alexander quandle. 
More recently, Ta studied the above problems when $X$ is a conjugation subquandle \cite{ta2025goodinvolutionsconjugationsubquandles}, a twisted conjugation subquandle, or an Alexander quandle \cite{ta2025goodinvolutionstwistedconjugation}.

\subsection{Main results}
In this paper, we address Problem \ref{prob: condition} and Problem \ref{prob: isomorphism class} for symplectic quandles, introduced by Navas--Nelson \cite{MR2493966}.

A \textit{symplectic quandle} $(M,*)$ is a quandle structure on a free module $M$ of finite rank over a commutative ring $R$ with identity equipped with an antisymmetric bilinear form.
A typical example is the symplectic quandle associated with a symplectic vector space.
Since the trivial quandle on any module over $R$ is also a symplectic quandle, we focus on symplectic quandles that are not the trivial quandle.
Such symplectic quandles are called \textit{nontrivial symplectic quandles}.

We first discuss the existence or nonexistence of $R$-linear good involutions on symplectic quandles.

\begin{thm} \label{thm: 1}
    Let $(M,*)$ be a nontrivial symplectic quandle over an integral domain $R$.
    Then the following are equivalent.
    \begin{itemize}
        \item There exists an $R$-linear good involution on $(M,*)$.
        \item $(M,*)$ is a kei.
        \item The characteristic of $R$ is $2$.
    \end{itemize}
\end{thm}

In particular, when the characteristic of an integral domain $R$ is not 2, we show the nonexistence of $R$-linear good involutions on nontrivial symplectic quandles over $R$.
Therefore, we answer Problem \ref{prob: condition} for nontrivial symplectic quandles with $R$-linear good involutions when $R$ is an integral domain.
Moreover, we compare linear good involutions with linear \textit{symplectic involutions} and \textit{anti-symplectic involutions}.

Note that $R$-linearity of good involutions is essential in Theorem \ref{thm: 1}.
In fact, we provide an example of a symplectic quandle that is not a kei but has a non-$R$-linear good involution (see Example~\ref{eg: good involution}).

When a symplectic quandle $(M,*)$ is a kei (equivalently, when the characteristic of $R$ is $2$), we discuss the nonexistence of good involutions on $(M,*)$ other than the identity map.
\begin{thm} \label{thm: 2}
    Let $(M,*)$ be a nontrivial symplectic quandle over an integral domain $R$.
    Suppose that the characteristic of $R$ is $2$.
    If the antisymmetric bilinear form on $M$ is nondegenerate, then the identity map of $M$ is the only good involution on $(M,*)$.
\end{thm}

Note that Theorem \ref{thm: 2} does not assume that good involutions are $R$-linear.
Thus, we answer Problem \ref{prob: isomorphism class} for symplectic quandles over an integral domain $R$ of characteristic $2$ with a nondegenerate antisymmetric bilinear form.
While we can show that symplectic quandles with a nondegenerate antisymmetric bilinear form are \textit{faithful} and Theorem \ref{thm: 2} follows from \cite[Corollary 5.5]{ta2025goodinvolutionsconjugationsubquandles}, our proof is based on the module-theoretic structure of symplectic quandles.

Suppose that $R$ is a commutative ring with identity and the characteristic of $R$ is not $2$. We also study the nonexistence of good involutions on symplectic quandles $(M,*)$ when the underlying module $M$ admits a \textit{hyperbolic pair} (see, e.g., \cite[page 586]{MR1878556}). This gives an answer to Problem \ref{prob: isomorphism class} for such symplectic quandles.

\begin{thm} \label{thm: 3}
    Let $(M,*)$ be a symplectic quandle over a commutative ring $R$ with identity.
    Suppose that the characteristic of $R$ is not $2$.
    If there exists a \textit{hyperbolic pair} in the $R$-module $M$, then there exists no good involution on $(M,*)$.
\end{thm}
In particular, there exists no good involution on the symplectic quandle $(V,*)$ over any symplectic vector space $V$.

Note that this theorem cannot answer Problem \ref{prob: condition} completely and there is a counterexample to the converse of Theorem \ref{thm: 3}.
In fact, we give an example of a symplectic quandle that admits no good involution and no hyperbolic pair (see Example~\ref{eg: no good involution but no hyperbolic pair}).

\subsection{Outline}
This paper is organized as follows.
In Section \ref{sec: quandle}, we provide preliminaries on quandles and good involutions.
In Section \ref{sec: symplectic quandle}, we recall the definition of symplectic quandles and show some basic properties.
In Section \ref{sec: linear good involution}, we show Theorem \ref{thm: 1}.
In Section \ref{sec: good involution char 2}, we show Theorem \ref{thm: 2}.
In Section \ref{sec: good involution char not 2}, we show Theorem \ref{thm: 3} and give counterexamples to the converse of Theorem \ref{thm: 3}.

\section{Quandles} \label{sec: quandle}

\begin{defin}[{\cite{MR638121,MR672410}}] \label{def: quandle}
Let $X$ be a set equipped with a binary operation $*:X \times X \to X; (x,y) \mapsto x*y$.
A pair $(X,*)$ is called a \textit{quandle} if the binary operation $*$ satisfies the following three conditions.
\begin{enumerate}
    \item $x*x = x$ for all $x \in X$,
    \item for any $x,y \in X$, there exists a unique element $z \in X$ such that $z*y = x$,
    \item $(x*y)*z = (x*z)*(y*z)$ for all $x,y,z \in X$.  
\end{enumerate}
\end{defin}

Quandles were introduced to define invariants of knots and links.
These conditions are motivated by the Reidemeister moves in knot theory.

\begin{remark} \label{remark: dual quandle}
    Let $(X,*)$ be a quandle.
    By the second condition, a pair $(X,*^{-1})$ is also a quandle, where the binary operation $*^{-1}:X \times X \to X$ is defined by $(x*y)*^{-1}y = x$ and $(x*^{-1}y)*y = x$.
    $(X,*^{-1})$ is called the \textit{dual quandle} of $(X,*)$ and $*^{-1}$ is called the \textit{dual operation}.
\end{remark}

\begin{defin}[cf.~\cite{MR21002}] \label{def: kei}
    A quandle $(X,*)$ is a \textit{kei} or an \textit{involutory quandle} if $x*y = x*^{-1}y$ for all $x,y \in X$.
\end{defin}

\begin{example} \label{eg: trivial quandle}
    Let $X$ be a nonempty set. 
    Define a binary operation $*:X \times X \to X$ by $x*y \coloneqq x$. 
    The pair $(X,*)$ is a quandle, called the \textit{trivial quandle} on $X$. 
    Any trivial quandle is a kei.
\end{example}

\begin{defin}[{\cite{MR2371714,MR2657689}}] \label{def: good involution}
    Let $(X,*)$ be a quandle.
    An involution $\rho:X \to X$ is a \textit{good involution} if it satisfies the following conditions.
    \begin{enumerate}
        \item $\rho(x*y) = \rho(x)*y$ for all $x,y \in X$,
        \item $x*\rho(y) = x *^{-1} y$ for all $x,y \in X$.
    \end{enumerate}
    A pair $(X, \rho)$ consisting of a quandle and a good involution on it is called a \textit{symmetric quandle}.
\end{defin}

Symmetric quandles were introduced to define invariants of knots that are not necessarily oriented or orientable.

\begin{example} \label{eg: good involution of trivial quandle}
    Any involution on a trivial quandle is a good involution.
\end{example}

\begin{example} \label{eg: good involution of kei}
    The identity map on a kei is a good involution.
\end{example}

Ta \cite{ta2025goodinvolutionsconjugationsubquandles} discussed the existence of good involutions on racks (quandles without the idempotency condition $x*x = x$).
We recall several results relevant to this paper.

\begin{defin}[{\cite[Definition 3.2]{ta2025goodinvolutionsconjugationsubquandles}}]
    Let $(X,*)$ be a rack. 
    An \textit{antiautomorphism} of $(X,*)$ is a rack isomorphism $\phi:(X,*) \to (X,*^{-1})$.
    Let $\operatorname{Anti}(X)$ denote the set of antiautomorphisms of $(X,*)$. A rack $(X,*)$ is \textit{self-dual} if $\operatorname{Anti}(X) \neq \emptyset$.
\end{defin}

Involutory racks are self-dual. Moreover, the following holds.

\begin{lemma}[{\cite[Example 3.4]{ta2025goodinvolutionsconjugationsubquandles}}]
    If a rack $(X,*)$ is involutory, then the following statements hold.
    \begin{itemize}
        \item $(X,*)$ is self-dual.
        \item The identity map $\operatorname{id}_{X}$ is an antiautomorphism of $(X,*)$.
        \item $\operatorname{Anti}(X) = \operatorname{Aut}(X)$, where $\operatorname{Aut}(X)$ is the set of rack automorphisms of $(X,*)$.
    \end{itemize}
\end{lemma}

\begin{lemma}
    Any good involution on a rack $(X,*)$ is an antiautomorphism of $(X,*)$.
\end{lemma}
\begin{proof}
    Let $\rho:X \to X$ be a good involution. Then 
    \begin{equation*}
        \rho(x*y) = \rho(x)*y = \rho(x)*^{-1}\rho(y)
    \end{equation*}
    for all $x,y \in X$, which shows that $\rho$ is an antiautomorphism of $(X,*)$.
\end{proof}

\begin{prop}[{\cite[Corollary 4.4]{ta2025goodinvolutionsconjugationsubquandles}}] \label{prop: not self-dual implies no good involution}
    If a rack $(X,*)$ is not self-dual, then $(X,*)$ has no good involutions.
\end{prop}

\section{Symplectic quandles} \label{sec: symplectic quandle}

\begin{defin}
    Let $R$ be a commutative ring with identity and $M$ a free $R$-module of finite rank equipped with an antisymmetric bilinear form $\langle, \rangle:M \times M \to R$ such that $\langle x,x \rangle = 0$ for all $x \in M$.
    Let $M^{\vee}$ denote the dual module of $M$.
    An antisymmetric bilinear form $\langle, \rangle$ on $M$ is said to be \textit{nondegenerate} if the map $M \to M^{\vee}$ defined by $x \mapsto \langle x, -\rangle$ is injective.
    An antisymmetric bilinear form $\langle, \rangle$ on $M$ is said to be \textit{unimodular} if the map $M \to M^{\vee}$ defined by $x \mapsto \langle x, -\rangle$ is bijective.
\end{defin}

Note that if $R$ is not a field, then the nondegeneracy of an antisymmetric bilinear form does not necessarily imply its unimodularity.

\begin{defin}[{\cite{MR2493966}}]
    Let $R$ be a commutative ring with identity and $M$ a free $R$-module of finite rank equipped with an antisymmetric bilinear form $\langle, \rangle:M \times M \to R$ such that $\langle x,x \rangle = 0$ for all $x \in M$. 
    Define a binary operation $*:M\times M \to M$ by 
    \begin{equation*}
        x*y \coloneqq x + \langle x,y \rangle y.
    \end{equation*}
    The pair $(M,*)$ is called a \textit{symplectic quandle over $R$}.
    A \textit{nondegenerate symplectic quandle over $R$} is a symplectic quandle equipped with a nondegenerate antisymmetric bilinear form $\langle, \rangle$.
\end{defin}

\begin{remark}
    The dual quandle $(M,*^{-1})$ of a symplectic quandle $(M,*)$ is given by 
    \begin{equation*}
        x*^{-1}y = x - \langle x,y \rangle y.
    \end{equation*} 
\end{remark}

We obtain the following property.

\begin{prop} \label{prop: right-action is linear}
Let $(M,*)$ be a symplectic quandle over a commutative ring $R$ with identity. 
For every $y \in M$, define a map $s_{y}:M \to M$ by $s_{y}(x) \coloneqq x * y$.
Then the map $s_{y}$ is an invertible $R$-linear map.
\end{prop}
\begin{proof}
The invertibility of the map $s_{y}: M \to M$ follows from the quandle axioms.
For all $x, x^{\prime} \in M$, 
\begin{align*}
    s_{y}(x + x^{\prime}) 
    &= (x + x^{\prime})*y \\
    &= (x + x^{\prime}) + \langle (x + x^{\prime}), y \rangle y \\
    &= x + x^{\prime} + \langle x,y \rangle y + \langle x^{\prime},y \rangle y \\
    &= x*y + x^{\prime}*y \\
    &= s_{y}(x) + s_{y}(x^{\prime}).
\end{align*}
For all $x \in M$ and all $r \in R$,
\begin{align*}
    s_{y}(rx) 
    &= (rx) + \langle (rx),y \rangle y \\
    &= rx + r\langle x,y \rangle y \\
    &= r(x + \langle x,y \rangle y) \\
    &= r(x*y) \\
    &= r s_{y}(x).
\end{align*}
Thus, the map $s_{y}:M \to M$ is $R$-linear for all $y \in M$.
\end{proof}

Here is a typical example of a symplectic quandle. 

\begin{example} \label{eg: symplectic vector space}
    Let $V$ be an $\mathbb{R}$-vector space equipped with a nondegenerate antisymmetric bilinear form $\langle, \rangle:V \times V \to \mathbb{R}$. 
    The pair $(V,\langle,\rangle)$ is called a \textit{symplectic vector space}.
    A symplectic vector space $V$ therefore defines a nondegenerate symplectic quandle $(V,*)$ over $\mathbb{R}$.
\end{example}

A trivial quandle on a free $R$-module of finite rank is also an example of a symplectic quandle.

\begin{example}
    Let $M$ be a free $R$-module of finite rank. 
    Define a bilinear form $\langle, \rangle: M \times M \to R$ by 
    \begin{equation*}
        \langle x,y \rangle \coloneqq 0
    \end{equation*}
    for all $x,y \in M$.
    Then the pair $(M,*)$ is a symplectic quandle, where
    \begin{equation*}
        x*y = x + \langle x,y \rangle y = x.
    \end{equation*}
    This symplectic quandle is nothing but the trivial quandle.
\end{example}

\begin{example}
    Let $\{0\}$ denote the zero module over a commutative ring $R$ with identity.
    Then the trivial quandle $(\{0\},*)$ is a nondegenerate symplectic quandle.
\end{example}

\begin{remark}
    Throughout this paper, we assume for simplicity that a module $M$ is not the zero module.
    However, the statements in this paper remain valid for $M = \{0\}$ unless otherwise indicated.
\end{remark}

\begin{defin}\label{def: nontrivial symplectic quandle}
    A symplectic quandle is called \textit{nontrivial} if it is not a trivial quandle.
\end{defin}

The structure of symplectic quandles can be characterized in terms of the characteristic of the base ring $R$.

\begin{prop}[{\cite[Proposition 1]{MR2493966}}] \label{prop: kei if char2}
If $(M,*)$ is a symplectic quandle over $R$ of characteristic $2$, then $(M,*)$ is a kei.
\end{prop}

The converse of Proposition \ref{prop: kei if char2} holds under the additional assumption that $R$ is an integral domain and $(M,*)$ is a nontrivial symplectic quandle.

\begin{prop}
    Suppose that $R$ is an integral domain and the characteristic of $R$ is not $2$.
    If a symplectic quandle $(M,*)$ over $R$ is a kei, then $(M,*)$ is a trivial quandle.
\end{prop}
\begin{proof}
    Suppose that a symplectic quandle $(M,*)$ is a kei, i.e. $* = *^{-1}$.

    First, for any $x \in M$, we have 
    \begin{equation*}
        x*0 = x, \; x*^{-1}0 = x.
    \end{equation*}

    Next, for any $x \in M$ and $y \in M \setminus \{0\}$ we have
    \begin{equation*}
        x + \langle x,y \rangle y = x - \langle x,y \rangle y,
    \end{equation*}
    i.e. $2\langle x,y \rangle y = 0$. 
    Since $R$ is an integral domain and $M$ is a free module over $R$, $M$ is torsion-free. Hence $2\langle x,y \rangle y = 0$ with $y \neq 0$ implies $2 \langle x,y \rangle = 0$. Since the characteristic of $R$ is not $2$, it follows that $\langle x,y \rangle = 0$.

    By the definition of a symplectic quandle, we obtain
    \begin{equation*}
        x*y = x + \langle x,y \rangle y = x
    \end{equation*}
    for all $x,y \in M$.
    Therefore, $(M,*)$ is a trivial quandle.
\end{proof}

Since a trivial quandle is always a kei, the following holds.
\begin{cor}
    Suppose that $R$ is an integral domain and the characteristic of $R$ is not $2$.
    A symplectic quandle $(M,*)$ over $R$ is a kei if and only if $(M,*)$ is a trivial quandle.
\end{cor}

\begin{prop} \label{prop: kei iff char2 when integral domain}
    Suppose that $R$ is an integral domain. 
    A nontrivial symplectic quandle $(M,*)$ over $R$ is a kei if and only if the characteristic of $R$ is $2$.
\end{prop}
\begin{proof}
    Suppose that a nontrivial symplectic quandle $(M,*)$ is a kei, i.e. $* = *^{-1}$.
    Then we obtain $2\langle x,y \rangle y = 0$ for all $x,y \in M$.
    Since $M$ is a free module over an integral domain, it follows that $M$ is torsion-free.
    Since $(M,*)$ is not a trivial quandle, the characteristic of $R$ must be $2$.

    The converse follows from Proposition \ref{prop: kei if char2}.
\end{proof}

We now consider linear \textit{symplectic and anti-symplectic automorphisms} and the relationship with symplectic quandles.

\begin{defin}\label{def: symplectic and anti-symplectic}
    Let $M$ be a free module of finite rank over $R$ equipped with an antisymmetric bilinear form $\langle, \rangle:M \times M \to R$.
    An automorphism $\rho: M \to M$ is called a \textit{symplectic automorphism} if 
    \begin{equation*}
        \langle \rho(x),\rho(y) \rangle = \langle x,y \rangle 
    \end{equation*} 
    for all $x,y \in M$.
    An automorphism $\rho: M \to M$ is called an \textit{anti-symplectic automorphism} if 
    \begin{equation*}
        \langle \rho(x),\rho(y) \rangle = -\langle x,y \rangle 
    \end{equation*} 
    for all $x,y \in M$.
\end{defin}
When $R = \mathbb{R}$ and the bilinear form is nondegenerate, these notions agree with the usual notions of linear symplectic and anti-symplectic automorphisms of a symplectic vector space.

\begin{prop}\label{prop: symplectic involution is automorphism}
    Let $(M,*)$ be a symplectic quandle over $R$. An $R$-linear symplectic automorphism $\rho:M \to M$ is an automorphism of the symplectic quandle $(M,*)$.
\end{prop}
\begin{proof}
    For all $x,y \in M$, we have 
    \begin{align*}
        \rho(x*y) &= \rho(x + \langle x,y \rangle y) \\
        &= \rho(x) + \langle x,y \rangle \rho(y) \\
        &= \rho(x) + \langle \rho(x),\rho(y) \rangle \rho(y) \\
        &= \rho(x)*\rho(y).
    \end{align*}
    Thus $\rho$ is an automorphism of $(M,*)$.
\end{proof}
\begin{cor}
    Let $(M,*)$ be a symplectic quandle over $R$.
    The group of $R$-linear symplectic automorphisms of $M$ is a subgroup of the quandle automorphism group $\operatorname{Aut}(M,*)$.
\end{cor}
\begin{prop}\label{prop: self-dual}
    Let $(M,*)$ be a symplectic quandle over $R$. 
    An $R$-linear anti-symplectic automorphism $\rho:M \to M$ is an antiautomorphism of the symplectic quandle $(M,*)$.
\end{prop}
\begin{proof}
    For all $x,y \in M$, we have 
    \begin{align*}
        \rho(x*y) &= \rho(x + \langle x,y \rangle y) \\
        &= \rho(x) + \langle x,y \rangle \rho(y) \\
        &= \rho(x) - \langle \rho(x),\rho(y) \rangle \rho(y) \\
        &= \rho(x)*^{-1}\rho(y).
    \end{align*}
    Thus $\rho$ is an isomorphism between $(M,*)$ and $(M,*^{-1})$, i.e. an antiautomorphism of the symplectic quandle $(M,*)$.
\end{proof}

Although the following proposition is well known to experts in symplectic geometry, we provide the proof for the convenience of the reader.
\begin{prop} \label{prop: existence of linear anti-symplectic involution}
For every finite-dimensional symplectic vector space $(V,\langle,\rangle)$, there exists a linear anti-symplectic involution on $(V,\langle, \rangle)$.
\end{prop}
\begin{proof}
Let $2n \coloneqq \dim_{\mathbb{R}}V$.
By a skew-symmetric version of the Gram-Schmidt process, we can choose a basis $e_{1},\ldots,e_{n},f_{1},\ldots,f_{n}$ of $V$ such that 
\begin{equation*}
    \langle e_{i},e_{j} \rangle = 0, \;
    \langle f_{i},f_{j} \rangle = 0, \;
    \langle e_{i},f_{j} \rangle = \delta_{ij},
\end{equation*}
for all $i,j = 1,\ldots,n$ (see~\cite[Theorem 1.1]{MR1853077}).

We define a linear map $\rho \colon V \to V$ by
\begin{equation*}
    \rho(e_{i}) \coloneqq e_{i}, \; \rho(f_{i}) \coloneqq -f_{i}
\end{equation*}
for all $i = 1,\ldots,n$.
For all $a_{1},\ldots,a_{n},b_{1},\ldots,b_{n} \in \mathbb{R}$, 
\begin{align*}
    &\rho(a_{1}e_{1}+ \cdots + a_{n}e_{n} + b_{1}f_{1}+ \cdots + b_{n}f_{n})\\
    &= a_{1}\rho(e_{1}) + \cdots + a_{n} \rho(e_{n}) + b_{1}\rho(f_{1}) + \cdots + b_{n}\rho(f_{n})\\
    &= a_{1}e_{1}+ \cdots + a_{n}e_{n} - b_{1}f_{1} - \cdots - b_{n}f_{n}.
\end{align*}
Thus, we have $\rho^2(v) =  v$ for all $v \in V$, i.e. $\rho$ is an involution.

Moreover, since 
\begin{align*}
    &\langle \rho(e_{i}),\rho(e_{j}) \rangle 
    = \langle e_{i},e_{j} \rangle 
    = 0 
    = - \langle e_{i},e_{j} \rangle, \\
    &\langle \rho(e_{i}),\rho(f_{j}) \rangle 
    = \langle e_{i}, -f_{j} \rangle 
    = - \langle e_{i},f_{j} \rangle, \\
    &\langle \rho(f_{i}),\rho(f_{j}) \rangle
    = \langle -f_{i},-f_{j} \rangle 
    = 0
    = - \langle f_{i},f_{j} \rangle,
\end{align*}
it follows that $\langle \rho(v),\rho(w) \rangle = - \langle v,w \rangle$ holds for all $v,w \in V$.

Therefore, $\rho$ is a linear anti-symplectic involution.
\end{proof}

Since an anti-symplectic involution is an anti-symplectic automorphism, Proposition \ref{prop: self-dual} and Proposition \ref{prop: existence of linear anti-symplectic involution} imply the following:
\begin{cor} \label{cor: self-dual symplectic vector space}
    Let $V$ be a finite-dimensional symplectic vector space.
    Then the symplectic quandle $(V,*)$ is self-dual.
\end{cor}
Moreover, since Albers--Frauenfelder~\cite{MR3001556} showed that the space of linear anti-symplectic involutions on a symplectic vector space is diffeomorphic to the homogeneous space $\operatorname{GL}(n,\mathbb{R})\backslash \operatorname{Sp}(n)$, the set of antiautomorphisms of the associated symplectic quandle contains the homogeneous space $\operatorname{GL}(n,\mathbb{R})\backslash \operatorname{Sp}(n)$.

\section{Linear good involutions} \label{sec: linear good involution}

In this section, we show the nonexistence of linear good involutions on nontrivial symplectic quandles.
In particular, we consider linear involutions satisfying either the first or the second condition in the definition of a good involution.
Precisely speaking, these two conditions are closely related to \textit{symplectic involutions} or \textit{anti-symplectic involutions} in symplectic geometry.

Throughout this section, except in Corollary~\ref{cor: main 1}, we assume that $R$ is an integral domain and the characteristic of $R$ is not $2$.

\begin{lemma} \label{lem: linear semi-good involution}
    Let $(M,*)$ be a symplectic quandle over an integral domain $R$.
    If an $R$-linear map $\rho:M \to M$ satisfies either the first or the second condition in the definition of a good involution, then we have 
    \begin{equation*}
        \langle \rho(x),x \rangle = 0
    \end{equation*}
    for all $x \in M$.
\end{lemma}
\begin{proof}
    Suppose that $\rho$ satisfies the first condition.
    Since $x*x = x$, we have 
    \begin{equation*}
        \rho(x) = \rho(x*x) = \rho(x)*x = \rho(x) + \langle \rho(x),x \rangle x
    \end{equation*}
    for all $x \in M$, which implies that $\langle \rho(x),x \rangle x = 0$.
    Since $R$ is an integral domain and $M$ is a free $R$-module, it follows that $M$ is torsion-free. Thus, $\langle \rho(x),x \rangle x = 0$ implies that $\langle \rho(x),x \rangle = 0$.

    Suppose that $\rho$ satisfies the second condition.
    Note that since $\rho$ is $R$-linear, $\rho(0) = 0$;
    thus $\langle \rho(0),0 \rangle = 0$.
    Since $x*x = x$, we have 
    \begin{equation*}
        x = x*x = x*^{-1}\rho(x) = x - \langle x,\rho(x) \rangle \rho(x)
    \end{equation*}
    for all $x \in M$, which implies that $\langle \rho(x),x \rangle\rho(x) = 0$.
    If $\rho(x) = 0$, then $\langle \rho(x),x \rangle = 0$ holds trivially. If $\rho(x) \neq 0$, then $\langle \rho(x),x \rangle\rho(x) = 0$ implies $\langle \rho(x),x \rangle = 0$ by the torsion-freeness of $M$.
\end{proof}

\begin{prop} \label{prop: symplectic involution}
    Let $(M,*)$ be a symplectic quandle over an integral domain $R$.
    If an $R$-linear involution $\rho:M \to M$ satisfies the first condition of the definition of a good involution, then we have 
    \begin{equation} \label{eq: symplectic involution}
        \langle \rho(x),\rho(y) \rangle = \langle x,y \rangle
    \end{equation}
    for all $x,y \in M$.
\end{prop}
\begin{proof}
    By Lemma \ref{lem: linear semi-good involution}, we have 
    \begin{equation*}
        \langle \rho(\rho(x)+ y), \rho(x)+y \rangle = 0.
    \end{equation*}
    Since $\rho$ is $R$-linear, we have 
    \begin{align*}
        \langle \rho(\rho(x)+ y), \rho(x)+y \rangle 
        &= 
        \langle x+\rho(y), \rho(x)+y \rangle \\
        &=
        \langle x,\rho(x) \rangle + \langle x,y \rangle + \langle \rho(y),\rho(x) \rangle + \langle \rho(y),y \rangle \\
        &=
        \langle x,y \rangle + \langle \rho(y),\rho(x) \rangle.
    \end{align*}
    Thus we obtain Equation \eqref{eq: symplectic involution}.
\end{proof}

When we see the bilinear form $\langle, \rangle$ as a (pre)symplectic form and $R = \mathbb{R}$, involutions that satisfy Equation \eqref{eq: symplectic involution} are called symplectic involutions in the context of symplectic geometry.

\begin{prop} \label{prop: anti-symplectic involution}
    Let $(M,*)$ be a symplectic quandle over an integral domain $R$.
    If an $R$-linear involution $\rho:M \to M$ satisfies both the first and second conditions of a good involution, then we have 
    \begin{equation} \label{eq: anti-symplectic involution}
        \langle \rho(x),\rho(y) \rangle = -\langle x,y \rangle
    \end{equation}
    for all $x,y \in M$.
\end{prop}
\begin{proof}
    For all $x,y \in M$, we have 
    \begin{align*}
        x*y &= \rho(\rho(x)*y) \\
        &= \rho(\rho(x)*^{-1}\rho(y))\\
        &= \rho(\rho(x) - \langle \rho(x),\rho(y) \rangle \rho(y))\\
        &= x - \langle \rho(x),\rho(y) \rangle y,
    \end{align*}
    i.e. we obtain 
    \begin{equation*}
        \langle x,y \rangle y = -\langle \rho(x),\rho(y) \rangle y.
    \end{equation*}
    Since $M$ is a free module over an integral domain $R$, it follows that $M$ is torsion-free.
    Therefore, we have Equation \eqref{eq: anti-symplectic involution}.
\end{proof}

Involutions that satisfy Equation \eqref{eq: anti-symplectic involution} are called anti-symplectic involutions in the context of symplectic geometry.

From Proposition \ref{prop: symplectic involution} and Proposition \ref{prop: anti-symplectic involution}, it follows that if an $R$-linear involution $\rho:M \to M$ satisfies only the first condition of the definition of a good involution, then $\rho$ is a symplectic involution. If $\rho$ also satisfies the second condition, i.e. $\rho$ is a good involution, then $\rho$ is an anti-symplectic involution as well.

\begin{thm} \label{thm: linear good involution}
    Let $(M,*)$ be a symplectic quandle over an integral domain $R$.
    If the characteristic of $R$ is not $2$ and there exists an $R$-linear good involution $\rho:M \to M$, then $(M,*)$ is a trivial quandle.
\end{thm}
\begin{proof}
    Let $\rho:M \to M$ be an $R$-linear good involution. 
    By Proposition \ref{prop: symplectic involution} and Proposition \ref{prop: anti-symplectic involution}, we have
    \begin{equation*}
        \langle x,y \rangle = \langle \rho(x), \rho(y) \rangle 
        = - \langle x,y \rangle 
    \end{equation*}
    for all $x,y \in M$. 
    Since $R$ is an integral domain, the characteristic of $R$ is not 2 and $M$ is a free module over $R$, we obtain $\langle x,y \rangle = 0$ for all $x,y \in M$.
    Therefore, $(M,*)$ is a trivial quandle (see Example~\ref{eg: trivial quandle}).
\end{proof}
We show the first main result (Theorem \ref{thm: 1}).
\begin{cor} \label{cor: main 1}
    Let $(M,*)$ be a nontrivial symplectic quandle over an integral domain $R$.
    Then the following are equivalent.
    \begin{itemize}
        \item There exists an $R$-linear good involution on $(M,*)$.
        \item $(M,*)$ is a kei.
        \item The characteristic of $R$ is 2.
    \end{itemize}
\end{cor}
\begin{proof}
    By Proposition \ref{prop: kei iff char2 when integral domain},
    $(M,*)$ is a kei if and only if the characteristic of $R$ is 2.
    In particular, if $(M,*)$ is a kei, the identity map on $M$ is an $R$-linear good involution.

    By Theorem \ref{thm: linear good involution}, we see that if there exists an $R$-linear good involution on a nontrivial symplectic quandle $(M,*)$, then the characteristic of $R$ is 2.
\end{proof}

Moreover, we give an example of a symplectic quandle that is not a kei but admits a non-$R$-linear good involution.

\begin{example}\label{eg: good involution}
    Let $\mathbb{Z}[i] = \{a + bi \mid a,b \in \mathbb{Z}\}$ denote the Gaussian integers.
    Let $M = (\mathbb{Z}[i])^{2}$ be a free module of rank two over $R \coloneqq \mathbb{Z}[i]$.
    Define an antisymmetric bilinear form $\langle, \rangle:M \times M \to \mathbb{Z}[i]$ by 
    \begin{equation*}
        \langle (a,b),(c,d) \rangle \coloneqq 3(ad - bc).
    \end{equation*}
    This bilinear form is nondegenerate. Note that since $3$ is not an invertible element in $R$, the bilinear form is not unimodular.
    Then $(M,\langle,\rangle)$ defines a symplectic quandle $(M,*)$.

    This symplectic quandle $(M,*)$ admits a good involution.
\end{example}
\begin{proof}
    For every nonzero element $x \in M \setminus \{0\}$, define the number $v(x)$ by 
    \begin{equation*}
        v(x) \coloneqq \max \{n\geq 0 \mid x \in 3^{n}M \},
    \end{equation*}
    where $3^{n}M \coloneqq \{(3^n a,3^n b) \mid (a,b) \in M \}$ is a submodule of $M$.
    The number $v(x)$ is well-defined. 
    In fact, since the ring $\mathbb{Z}[i]$ is a Euclidean domain and the element $3$ is prime in $\mathbb{Z}[i]$, every nonzero element of $\mathbb{Z}[i]$ is divisible by $3$ only finitely many times.
    Thus $v(x)$ is well-defined for every $x \in M \setminus \{0\}$.

    For any nonzero $x \in M \setminus \{0\}$, there exists $u \in M \setminus 3M$ such that $x = 3^{v(x)}u$.
    Let $\bar{x}$ denote the image of $u$ under the projection $M \to M/3M$.
    Note that $\bar{x}$ is not, in general, the residue class of $x$ modulo $3M$.

    There is a natural isomorphism
    \begin{equation*}
        M/3M \cong (\mathbb{Z}[i]/3\mathbb{Z}[i])^{2}.
    \end{equation*}
    Now since we have $\mathbb{Z}[i] \cong \mathbb{Z}[X]/(X^2 +1)$, we obtain 
    \begin{equation*}
        \mathbb{Z}[i]/3\mathbb{Z}[i] 
        \cong
        (\mathbb{Z}/3\mathbb{Z})[X]/(X^2 +1)
        = 
        \mathbb{F}_{3}[X]/(X^2 +1).
    \end{equation*}
    Here, $\mathbb{F}_{3}$ is the finite field of order $3$.
    Since the polynomial $X^2 + 1$ has no root in $\mathbb{F}_{3}$, it is irreducible over $\mathbb{F}_{3}$.
    Thus $\mathbb{F}_{3}[X]/(X^2 +1)$ is a quadratic extension of $\mathbb{F}_{3}$ and is isomorphic to $\mathbb{F}_{9}$, i.e. $M/3M \cong (\mathbb{F}_{9})^{2}$.

    Let $i$ denote the image of $i \in \mathbb{Z}[i]$ under the projection $\mathbb{Z}[i] \to \mathbb{Z}[i]/3\mathbb{Z}[i] \cong \mathbb{F}_{9}$. 
    The equality $i^2 = -1$ remains valid in $\mathbb{F}_{9}$.

    For $a \in \mathbb{F}_{9}$, define a linear map $f_{a}:M/3M \to M/3M$ by $f_{a}(x) \coloneqq ax$.
    Since $1,-1,i,-i \in \mathbb{F}_{9}$ are units, the maps $f_1, f_{-1}, f_{i}, f_{-i}$ induce an action on $M/3M \cong (\mathbb{F}_{9})^{2}$. Moreover, the orbit of $v \in (M/3M) \setminus \{0\}$ is the set $\{v,-v,iv,-iv\} \subset M/3M \cong (\mathbb{F}_{9})^{2}$.
    
    Choose a representative $s$ from each nonzero orbit. The corresponding orbit is the four-element set $\{s,-s,is,-is\}$.
    Define a map $\sigma: (M/3M) \setminus \{0\} \to \{i,-i\} \subset \mathbb{Z}[i]$ by 
    \begin{equation*}
        \sigma(s) \coloneqq i, \; 
        \sigma(-s) \coloneqq i, \; 
        \sigma(is) \coloneqq -i, \; 
        \sigma(-is) \coloneqq -i.
    \end{equation*}
    Then for all $v \in (M/3M) \setminus \{0\}$ we have 
    \begin{equation} \label{eq: involutive}
        \sigma(v)^2 = -1
    \end{equation}
    and 
    \begin{equation} \label{eq: sigma map}
        \sigma(\sigma(v)v) = \sigma(v)^{-1}.
    \end{equation}
    Indeed, if $v = s$, then 
    \begin{equation*}
        \sigma(\sigma(s)s) = \sigma(is) = -i = i^{-1} = \sigma(s)^{-1}.
    \end{equation*}
    If $v = -s$, then 
    \begin{equation*}
        \sigma(\sigma(-s)(-s)) = \sigma(-is) = -i = i^{-1} = \sigma(-s)^{-1}.
    \end{equation*}
    If $v = is$, then 
    \begin{equation*}
        \sigma(\sigma(is)(is)) = \sigma(-i^2 s) = \sigma(s) = i = (-i)^{-1} = \sigma(is)^{-1}.
    \end{equation*}
    If $v = -is$, then 
    \begin{equation*}
        \sigma(\sigma(-is)(-is)) = \sigma((-i)(-is)) = \sigma(-s) = i = (-i)^{-1} = \sigma(-is)^{-1}.
    \end{equation*}
    Thus we obtain Equation \eqref{eq: sigma map}.
    Equation \eqref{eq: involutive} is immediate.

    Define a map $\rho:M \to M$ by 
    \begin{equation*}
        \rho(x) \coloneqq \begin{cases*}
            \sigma(\bar{x})x & if $x \neq 0$, \\
            0 & if $x = 0$.
        \end{cases*}
    \end{equation*}
    We will show that this map $\rho$ is a good involution on the symplectic quandle $(M,*)$.

    First, we show that $\rho$ is an involution.
    If $x = 0$, then 
    $\rho^2(0) = \rho(0) = 0$.
    Since $\sigma(\bar{x})$ is a unit in $\mathbb{Z}[i]$ for every $x \neq 0$, it follows that $v(\sigma(\bar{x})x) = v(x)$.
    Thus, $\overline{\rho(x)} = \sigma(\bar{x})\bar{x}$.
    By Equation \eqref{eq: sigma map}, we have 
    \begin{equation*}
        \rho^2(x) = \rho(\sigma(\bar{x})x) 
        = \sigma(\sigma(\bar{x})\bar{x})\sigma(\bar{x})x
        = \sigma(\bar{x})^{-1}\sigma(\bar{x})x
        = x,
    \end{equation*}
    i.e. $\rho^2(x) = x$ for all $x$. 

    Second, we show that $\rho(x*y) = \rho(x)*y$ for all $x,y \in M$.
    If $x = 0$, then $\rho(0*y) = 0 = \rho(0)*y$.
    If $x = (x_1,x_2) \neq 0$, then for all $y =(y_1,y_2) \in M$ we have 
    \begin{align*}
        x*y &= x + \langle x,y \rangle y \\
        &= x + 3 (x_1 y_2 - x_2 y_1) y \\
        &= 3^{v(x)}u + 3 (3^{v(x)}u_1 y_2 - 3^{v(x)}u_2 y_1) y \\
        &= 3^{v(x)} (u + 3 (u_1 y_2 - u_2 y_1) y).
    \end{align*}
    Since $u + 3 (u_1 y_2 - u_2 y_1) y = u$ in $M/3M$ and $u \neq 0$, it follows that $u + 3 (u_1 y_2 - u_2 y_1) y \neq 0$.
    Moreover, $v(x*y) = v(x)$ and $\overline{x*y} = \bar{x} = u$.
    Thus, we have 
    \begin{align*}
        \rho(x)*y &= (\sigma(\bar{x})x)*y \\
        &= \sigma(\bar{x})x + \langle \sigma(\bar{x})x,y \rangle y \\
        &= \sigma(\bar{x}) (x*y) \\
        &= \sigma(\overline{x*y}) (x*y) \\
        &= \rho(x*y),
    \end{align*}
    i.e. we have $\rho(x*y) = \rho(x)*y$ for all $x,y \in M$.

    Finally, we show that $x*\rho(y) = x*^{-1}y$ for all $x,y \in M$.
    If $y = 0$, then $x*\rho(0) = x*0 = x = x*^{-1}0$.
    If $y \neq 0$, then we have 
    \begin{align*}
        x * \rho(y) &= x + \langle x,\rho(y) \rangle \rho(y) \\
        &= x + \langle x, \sigma(\bar{y})y \rangle \sigma(\bar{y})y \\
        &= x + \sigma(\bar{y})^2 \langle x, y \rangle y \\
        &= x - \langle x,y \rangle y \\
        &= x*^{-1}y,
    \end{align*}
    i.e. we have $x*\rho(y) = x*^{-1}y$ for all $x,y \in M$.

    Therefore, the map $\rho:M \to M$ is a good involution on $(M,*)$.
\end{proof}

\section{Good involutions when the characteristic is 2} \label{sec: good involution char 2}

We have seen in Proposition \ref{prop: kei iff char2 when integral domain} that if $R$ is an integral domain and the characteristic of $R$ is $2$, then any symplectic quandle $(M,*)$ over $R$ is a kei.
In this case, the identity map $\operatorname{id}_{M}$ on $M$ is a good involution on the symplectic quandle $(M,*)$.
In this section, we discuss the nonexistence of good involutions on such symplectic quandles other than the identity map.
In particular, we show that if an antisymmetric bilinear form $\langle, \rangle: M \times M \to R$ on $M$ is nondegenerate, then the set of all good involutions on $(M,*)$ is $\{\operatorname{id}_{M}\}$.

Note that the main result (Theorem \ref{thm: main 2}) 
also follows by combining the faithfulness of nondegenerate symplectic quandles with Ta's result~\cite[Corollary 5.5]{ta2025goodinvolutionsconjugationsubquandles}.
Our proof instead uses the module structures of symplectic quandles.

\begin{lemma} \label{lem: good involution sends zero to zero char2}
    Let $R$ be an integral domain of characteristic $2$.
    Let $(M,*)$ be a symplectic quandle over $R$.
    If an antisymmetric bilinear form $\langle, \rangle: M \times M \to R$ on $M$ is nondegenerate, then $\rho(0) = 0$ for every good involution $\rho:M \to M$ on $(M,*)$.
\end{lemma}
\begin{proof}
    Let $\rho:M \to M$ be a good involution on $(M,*)$.
    By the second condition of the definition of a good involution, we have 
    \begin{equation} \label{eq: good involution 2nd when char 2}
        \langle x,\rho(y) \rangle \rho(y) = \langle x,y \rangle y
    \end{equation}
    for all $x,y \in M$.

    Assume on the contrary that $\rho(0) \neq 0$.
    Then by the nondegeneracy of the bilinear form $\langle, \rangle$ on $M$,
    there exists $z \in M$ such that $\langle z,\rho(0) \rangle \neq 0$.
    From Equation \eqref{eq: good involution 2nd when char 2}, we obtain
    \begin{equation}\label{eq: zero to nonzero}
        \langle z,\rho(0) \rangle \rho(0) = \langle z,0 \rangle 0 = 0.
    \end{equation}
    Since $M$ is a free module over an integral domain $R$, $M$ is torsion-free.
    Thus $\langle z,\rho(0) \rangle \neq 0$ and $\rho(0) \neq 0$ imply that $\langle z,\rho(0) \rangle \rho(0) \neq 0$, which contradicts Equation \eqref{eq: zero to nonzero}.
\end{proof}

\begin{lemma} \label{lem: good involution implies expansion char2}
    Suppose that $R$ is an integral domain and the characteristic of $R$ is $2$.
    Let $(M,*)$ be a symplectic quandle over $R$.
    If an antisymmetric bilinear form $\langle, \rangle: M \times M \to R$ on $M$ is nondegenerate, then for any good involution $\rho:M \to M$ on $(M,*)$ and any nonzero element $y \in M$, there exists $\lambda$ in the fraction field $\operatorname{Frac}(R)$ of $R$ such that $\rho(y) = \lambda y$ in $M\otimes_{R}\operatorname{Frac}(R)$. 
\end{lemma}
\begin{proof}
    Since the antisymmetric bilinear form $\langle, \rangle$ on $M$ is nondegenerate, for every $y \in M \setminus \{0\}$, there exists $x \in M\setminus \{0\}$ such that $\langle x,y \rangle \neq 0$.
    By Equation \eqref{eq: good involution 2nd when char 2}, we obtain 
    \begin{equation*}
        \langle x, \rho(y) \rangle \rho(y) = \langle x,y \rangle y \neq 0.
    \end{equation*}
    Since $M$ is a free module over an integral domain $R$, it follows that $M$ is torsion-free. Thus, $\langle x, \rho(y) \rangle \neq 0$ for such $x,y \in M\setminus \{0\}$.

    Since $R$ is an integral domain, $\langle x, \rho(y) \rangle$ is invertible in the fraction field $\operatorname{Frac}(R)$ of $R$. Thus from Equation \eqref{eq: good involution 2nd when char 2} we obtain 
    \begin{equation*}
        \rho(y) = \langle x, \rho(y) \rangle^{-1} \langle x,y \rangle y,
    \end{equation*}
    in $M\otimes_{R}\operatorname{Frac}(R)$, 
    i.e. there exists $\lambda \in \operatorname{Frac}(R)$ such that $\rho(y) = \lambda y$.
\end{proof}
We show the second main result (Theorem \ref{thm: 2}).
\begin{thm} \label{thm: main 2}
    Let $R$ be an integral domain of characteristic $2$.
    Let $(M,*)$ be a symplectic quandle over $R$.
    If an antisymmetric bilinear form $\langle, \rangle: M \times M \to R$ on $M$ is nondegenerate, then the set of all good involutions on $(M,*)$ is $\{\operatorname{id}_{M}\}$.
\end{thm}
\begin{proof}
    Let $\rho:M \to M$ be a good involution on $(M,*)$.
    By Lemma \ref{lem: good involution implies expansion char2}, for any nonzero element $y \in M$ there exists $\lambda \in \operatorname{Frac}(R)$ such that $\rho(y) = \lambda y$ in $M\otimes_{R}\operatorname{Frac}(R)$.

    By the nondegeneracy of the antisymmetric bilinear form $\langle, \rangle$, it follows that for any nonzero element $y \in M$, there exists an element $x \in M$ such that $\langle x,y \rangle \neq 0$.
    Substituting $x,y$ in Equation \eqref{eq: good involution 2nd when char 2}, we obtain
    \begin{equation*}
        \lambda^{2} \langle x,y \rangle y = \langle x,y \rangle y.
    \end{equation*}
    Since $M$ is a free module and $y$ is a nonzero element, at least one coordinate of $y$ with respect to a fixed basis is nonzero. For simplicity, suppose that the first component $y_{1}$ of $y$ is nonzero. Then we obtain from the above equation that 
    \begin{equation*}
        \lambda^2 \langle x,y \rangle y_{1} = \langle x,y \rangle y_{1}.
    \end{equation*}
    Since $R$ is an integral domain, $\langle x,y \rangle \neq 0$ and $y_{1} \neq 0$, we obtain $\lambda^{2} = 1$.

    Note that since the characteristic of $R$ is $2$, the characteristic of $\operatorname{Frac}(R)$ is $2$. Since 
    \begin{equation*}
        \lambda^2 +1 = \lambda^2 + 2\lambda +1 = (\lambda +1)^{2} = 0,
    \end{equation*}
    it follows that $(\lambda +1)^{2} = 0$, i.e. $\lambda = 1$.
    Thus by Lemma \ref{lem: good involution implies expansion char2}, we have $\rho(y) = \lambda y = y$ for all nonzero elements $y \in M$, i.e. $\rho\mid_{M \setminus \{0\}} = \operatorname{id}_{M \setminus \{0\}}$.
    Moreover, by Lemma \ref{lem: good involution sends zero to zero char2}, we have $\rho = \operatorname{id}_{M}$.
\end{proof}

The nondegeneracy of an antisymmetric bilinear form $\langle, \rangle$ on a free $R$-module $M$ is essential for Theorem \ref{thm: main 2} at least when $R$ is a field of characteristic 2. We provide an example.

Let $K$ be a field, $V$ a finite-dimensional vector space over $K$ and $\langle, \rangle \colon V \times V \to K$ an antisymmetric bilinear form.
Let 
\begin{equation*}
    \operatorname{rad}(V) = \{
        v \in V \mid \langle v,w \rangle = 0
        \text{ for all } w \in V \}
\end{equation*}
be the \textit{radical} of $V$. 
Then there exists a complementary subspace $W$ to $\operatorname{rad}(V)$ in $V$ such that
\begin{equation*}
    V = \operatorname{rad}(V)\oplus W,
\end{equation*}
(see, e.g., \cite[Theorem 1.1]{MR1853077}).

\begin{example}
Let $\mathbb{F}_{2}$ denote the finite field of order $2$, 
and $V \coloneqq (\mathbb{F}_{2})^{3}$.
We define an antisymmetric bilinear form $\langle, \rangle \colon V \times V \to \mathbb{F}_{2}$ by 
\begin{equation*}
    \langle (x_{1},x_{2},x_{3}),(y_{1},y_{2},y_{3}) \rangle \coloneqq x_{1}y_{2}+x_{2}y_{1}.
\end{equation*}
Then we have $\operatorname{rad}(V) = \{(0,0,a) \mid a \in \mathbb{F}_{2}\}$.
We define a map $\rho \colon V \to V$ by 
\begin{equation*}
\rho(x) \coloneqq 
\begin{cases*}
    (0,0,a+1) & if $x = (0,0,a) \in \operatorname{rad}(V)$,\\ 
    x & if $x \notin \operatorname{rad}(V)$.
\end{cases*}
\end{equation*}
The map $\rho$ is a good involution on the associated symplectic quandle $(V,\ast)$ satisfying $\rho|_{V \setminus \operatorname{rad}(V)} = \operatorname{id}$ and $\rho|_{\operatorname{rad}(V)} \neq \operatorname{id}$.
\end{example}
\begin{proof}
First, we show that $u \ast v \in V \setminus \operatorname{rad}(V)$ for all $u,v \in V \setminus \operatorname{rad}(V)$.
Note that since $(0,0,a) \in \operatorname{rad}(V)$ for all $a \in \mathbb{F}_{2}$, it follows that 
\begin{equation*}
    (0,0,a) \ast v = (0,0,a), \; v \ast (0,0,a) = v
\end{equation*}
for all $v \in V$.
The subset $V \setminus \operatorname{rad}(V)$ consists of these elements:
\begin{equation*}
    V \setminus \operatorname{rad}(V)
    =
    \{
        (1,0,x),(0,1,y),(1,1,z) \mid x,y,z \in \mathbb{F}_{2}
    \}.
\end{equation*}
Direct computation shows that 
\begin{align*}
    &(1,0,x) \ast (0,1,y) = (1,0,x) + (0,1,y) = (1,1,x+y), \\
    &(1,0,x) \ast (1,1,z) = (1,0,x) + (1,1,z) = (0,1,x+z), \\
    &(0,1,y) \ast (1,0,x) = (0,1,y) + (1,0,x) = (1,1,x+y), \\
    &(0,1,y) \ast (1,1,z) = (0,1,y) + (1,1,z) = (1,0,y+z), \\
    &(1,1,z) \ast (1,0,x) = (1,1,z) + (1,0,x) = (0,1,x+z), \\
    &(1,1,z) \ast (0,1,y) = (1,1,z) + (0,1,y) = (1,0,y+z).
\end{align*}
Moreover, for elements of the same type, we have
\begin{align*}
    &(1,0,x) \ast (1,0,x') = (1,0,x),\\
    &(0,1,y) \ast (0,1,y') = (0,1,y),\\
    &(1,1,z) \ast (1,1,z') = (1,1,z),
\end{align*}
for all $x,x',y,y',z,z' \in \mathbb{F}_{2}$.
Indeed, the corresponding pairings are all zero.
Thus, $u \ast v \notin \operatorname{rad}(V)$ for all $u,v \notin \operatorname{rad}(V)$.
In particular, if $u \notin \operatorname{rad}(V)$, then $u \ast v \notin \operatorname{rad}(V)$ holds for all $v \in V$.

Next, we show that $\rho$ is an involution.
If $x = (0,0,a) \in \operatorname{rad}(V)$, then
\begin{equation*}
    \rho^2(0,0,a) = \rho(0,0,a+1) = (0,0,a+1+1) = (0,0,a).
\end{equation*}
If $x \notin \operatorname{rad}(V)$, then $\rho|_{V \setminus \operatorname{rad}(V)} = \operatorname{id}$ implies that 
\begin{equation*}
    \rho^{2}(x) = \rho(\operatorname{id}(x)) = \rho(x) = x.
\end{equation*}
Thus, $\rho$ is an involution on $V$.

Now we show that $\rho$ is a good involution.
First, we show that $\rho(u)\ast v = \rho(u\ast v)$ for all $u,v \in V$.
If $u \in \operatorname{rad}(V)$, then $u \ast v = u$, which yields
\begin{align*}
    \rho(0,0,a) \ast (v_{1},v_{2},v_{3}) &= (0,0,a+1) \ast (v_{1},v_{2},v_{3}) \\
    & = (0,0,a+1) \\
    &= \rho(0,0,a).
\end{align*}
If $u \notin \operatorname{rad}(V)$, then since $u\ast v \notin \operatorname{rad}(V)$, it follows that 
\begin{equation*}
    \rho(u) \ast v = u \ast v = \rho(u \ast v)
\end{equation*}
holds for all $u \notin \operatorname{rad}(V)$.
Thus, $\rho(u \ast v) = \rho(u) \ast v$ holds for all $u,v \in V$.

Next, we show that $u\ast \rho(v) = u \ast^{-1} v (= u \ast v)$ for all $u,v \in V$.
For all $a \in \mathbb{F}_{2}$ and $u = (u_{1},u_{2},u_{3}) \in V$, we have 
\begin{align*}
    (u_{1},u_{2},u_{3})\ast \rho(0,0,a) 
    &= (u_{1},u_{2},u_{3})\ast (0,0,a+1) \\
    &= (u_{1},u_{2},u_{3}) \\
    &= (u_{1},u_{2},u_{3})\ast (0,0,a).
\end{align*}
If $v \notin \operatorname{rad}(V)$, then $\rho(v) = v$ implies $u \ast \rho(v) = u \ast v$.
Thus, $u \ast \rho(v) = u \ast v$ holds for all $u,v \in V$.

Therefore, $\rho$ is a good involution on $(V,\ast)$.
\end{proof}
\section{Good involutions when the characteristic is not 2} \label{sec: good involution char not 2}

Let $R$ be a commutative ring with identity. We assume that the characteristic of $R$ is not $2$.
In this section, we discuss the nonexistence of good involutions on symplectic quandles over such $R$.

\begin{defin}\label{def: hyperbolic pair}
    Let $M$ be a free module of finite rank over a commutative ring $R$ with identity, and $\langle, \rangle:M \times M \to R$ an antisymmetric bilinear form on $M$.
    An ordered pair $(e,f)$ of two elements $e,f \in M$ is called a \textit{hyperbolic pair} if $\langle e,f \rangle = 1$.
\end{defin}
Note that if $(e,f)$ is a hyperbolic pair of a free module $M$ over $R$, then both $e$ and $f$ are nonzero elements of $M$.
\begin{lemma}\label{lem: good involution implies expansion hyperbolic pair}
    Let $M$ be a free module of finite rank over a commutative ring $R$ with identity and $\langle, \rangle:M \times M \to R$ an antisymmetric bilinear form on $M$.
    If $M$ admits a hyperbolic pair $(e,f)$, then for any good involution $\rho:M \to M$ on the symplectic quandle $(M,*)$, there exists $\lambda \in R$ such that $\rho(e) = \lambda e$ and $\lambda^2 = -1$.
\end{lemma}
\begin{proof}
    By the second condition of the definition of a good involution, we obtain 
    \begin{equation*}
        \langle x,\rho(y) \rangle \rho(y) = -\langle x,y \rangle y
    \end{equation*}
    for all $x,y \in M$.

    Let $(e,f)$ be a hyperbolic pair of $M$. Then 
    \begin{equation*}
        \langle f,\rho(e) \rangle \rho(e) = - \langle f,e \rangle e = e
    \end{equation*}
    for this hyperbolic pair $(e,f)$.
    Define $c \coloneqq \langle f,\rho(e) \rangle$. Then 
    \begin{equation} \label{eq: hyperbolic pair}
        c \rho(e) = e.
    \end{equation}
    Since $(e,f)$ is a hyperbolic pair, we obtain 
    \begin{equation*}
        1 = \langle e,f \rangle
        = \langle c\rho(e),f \rangle 
        = c \langle \rho(e),f \rangle 
        = c(-\langle f,\rho(e) \rangle)
        = c(-c)
        = -c^2.
    \end{equation*}
    In particular, $c$ is an invertible element of $R$ and $c^{-1} = -c$.
    From Equation \eqref{eq: hyperbolic pair}, we have 
    \begin{equation*}
        \rho(e) = c^{-1}e = -ce.
    \end{equation*}
    Define $\lambda \coloneqq -c$. Then we obtain $\rho(e) = \lambda e$ and $\lambda^{2} = (-c)^2 = -1$.
\end{proof}
\begin{lemma}\label{lem: good involution and linearity hyperbolic pair}
    Let $\rho:M \to M$ be a good involution on a symplectic quandle $(M,*)$.
    Fix $x \in M$. If there exists $\lambda_x \in R$ such that $\rho(x) = \lambda_{x}x$, then $\rho(x*y) = \lambda_{x}(x*y)$ holds for every $y \in M$.
\end{lemma}
\begin{proof}
    By the linearity of the map $s_{y}:M \to M; x \mapsto x*y$ (see Proposition~\ref{prop: right-action is linear}), we obtain $\rho(x)*y = (\lambda_{x}x)*y = \lambda_{x}(x*y)$.
    By the first condition of the definition of a good involution, we obtain 
    \begin{equation*}
        \rho(x*y) = \rho(x)*y = \lambda_{x}(x*y).
    \end{equation*} 
    Therefore, we obtain $\rho(x*y) = \lambda_{x}(x*y)$.
\end{proof}
We show the third main result (Theorem \ref{thm: 3}).
\begin{thm}\label{thm: main 3}
    Let $M$ be a free module of finite rank over a commutative ring $R$ with identity.
    Suppose that the characteristic of $R$ is not $2$.
    If $M$ admits a hyperbolic pair, then there exists no good involution on the symplectic quandle $(M,*)$.
\end{thm}
\begin{proof}
    Assume on the contrary that there exists a good involution $\rho:M \to M$ on the symplectic quandle $(M,*)$.

    Let $(e,f)$ be a hyperbolic pair of $M$.
    Then Lemma \ref{lem: good involution implies expansion hyperbolic pair} implies that there exists $\lambda \in R$ such that $\rho(e) = \lambda e$ and $\lambda^2 = -1$.

    Define the element $u_{1}\coloneqq e * f$. 
    Since $\langle e,f \rangle = 1$, we have 
    \begin{equation*}
        u_{1} = e + \langle e,f \rangle f = e + f.
    \end{equation*}
    Define $y \coloneqq e + \lambda f$. Since 
    \begin{equation*}
        \langle u_1, y \rangle 
        = \langle e+f, e + \lambda f \rangle 
        = \lambda \langle e,f \rangle + \langle f,e \rangle 
        = \lambda -1,
    \end{equation*}
    we obtain 
    \begin{equation*}
        u_{2} \coloneqq u_{1} * y 
        = u_{1} + \langle u_{1},y \rangle y 
        = e+f + (\lambda -1)(e + \lambda f)
        = \lambda e - \lambda f.
    \end{equation*}
    We further obtain 
    \begin{equation*}
        u_{2}* f 
        = \lambda e - \lambda f + \langle \lambda e - \lambda f, f \rangle f 
        = \lambda e - \lambda f + \lambda f 
        = \lambda e,
    \end{equation*}
    i.e. $\lambda e = u_{2}* f$.

    Next, we show that $\rho(\lambda e) = \lambda^2 e$. 
    By the definition of a good involution and the $R$-linearity of the maps $s_{y}$ from Proposition \ref{prop: right-action is linear}, we have 
    \begin{align*}
        \rho(\lambda e)
        &= \rho(u_{2}* f) \\
        &= \rho(u_{2})*f \\
        &= ((\rho(e)*f)*y)*f \\
        &= (((\lambda e)*f)*y)*f \\
        &= \lambda((e*f)*y)*f \\
        &= \lambda (u_{2}* f) \\
        &= \lambda^2 e.
    \end{align*}

    Since $\rho$ is an involution, $e = \rho^2(e)$.
    Since $\rho^2(e) = \rho(\lambda e) = \lambda^2 e$ and $(e,f)$ is a hyperbolic pair, we have $e = \lambda^2 e$. Thus we obtain 
    \begin{equation*}
        \langle e,f \rangle = \lambda^2 \langle e,f \rangle,
    \end{equation*}
    i.e. $\lambda^2 = 1$, but by Lemma \ref{lem: good involution implies expansion hyperbolic pair}, $\lambda^2 =-1$.
    This contradicts the assumption that the characteristic of $R$ is not $2$.
\end{proof}
It is well known in symplectic geometry that every symplectic vector space admits a symplectic basis (see, e.g., \cite[Theorem 1.1]{MR1853077}). 
Such a symplectic basis contains a hyperbolic pair; hence the associated symplectic quandle admits no good involution.
Corollary \ref{cor: self-dual symplectic vector space} implies that such symplectic quandles are examples of self-dual quandles that have no good involutions, i.e. they are counterexamples to the converse of Proposition \ref{prop: not self-dual implies no good involution}.

In general, the following holds.
\begin{cor}
    Let $K$ be a field and let $V$ be a finite-dimensional vector space over $K$. 
    If the characteristic of $K$ is not $2$ and the antisymmetric bilinear form $\langle, \rangle: V \times V \to K$ is nontrivial, i.e. there exist $x,y \in V$ such that $\langle x,y \rangle \neq 0$,
    then there exists no good involution on the symplectic quandle $(V,*)$.
\end{cor}
\begin{proof}
    Suppose that there exist $x,y \in V$ such that $\langle x,y \rangle \neq 0$. 
    Define $a \coloneqq \langle x,y \rangle$.
    Then the pair $(a^{-1}x,y)$ is a hyperbolic pair.
    Thus Theorem \ref{thm: main 3} implies the nonexistence of good involutions.
\end{proof}
\begin{cor}
    Let $R$ be a commutative ring with identity and $M \neq \{0\}$ a free module of finite rank over $R$.
    If the characteristic of $R$ is not $2$ and the antisymmetric bilinear form $\langle, \rangle: M \times M \to R$ is unimodular, then there exists no good involution on the symplectic quandle $(M,*)$.
\end{cor}
\begin{proof}
    By the unimodularity of the antisymmetric bilinear form $\langle, \rangle: M \times M \to R$, $M$ admits a hyperbolic pair.
    Thus Theorem \ref{thm: main 3} implies the nonexistence of good involutions.
\end{proof}

We give a counterexample to the converse of Theorem \ref{thm: main 3}, namely, a symplectic quandle whose underlying module has no hyperbolic pair and which admits no good involution.

\begin{example} \label{eg: no good involution but no hyperbolic pair}
    Let $R \coloneqq \mathbb{Z}/9\mathbb{Z}$ and let $M \coloneqq (R)^2$ be a free module of rank two over $R$.
    Define an antisymmetric bilinear form $\langle, \rangle: M \times M \to R$ by 
    \begin{equation*}
        \langle (a,b), (c,d) \rangle 
        \coloneqq 
        3 (ad - bc).
    \end{equation*}
    Then $(M,\langle,\rangle)$ defines a symplectic quandle $(M,*)$.
    
    This symplectic quandle $(M,*)$ admits no good involution, but the module $M$ equipped with $\langle, \rangle$ admits no hyperbolic pair. 
\end{example}
\begin{proof}
    First, we show that there is no hyperbolic pair of $M$, i.e. $\langle x,y \rangle \neq 1$ for all $x,y \in M$.

    Assume on the contrary that $(x,y)$ is a hyperbolic pair of $M$.
    By the definition of the bilinear form $\langle, \rangle$ on $M$, there exists $k \in R = \mathbb{Z}/9\mathbb{Z}$ such that $\langle x,y \rangle = 3k$, i.e. $1 = 3k$. This contradicts the fact that $3 \in R$ is not an invertible element of $R =\mathbb{Z}/9\mathbb{Z}$.
    Thus there exists no hyperbolic pair of $M$.

    Second, we show that the symplectic quandle $(M,*)$ admits no good involution.

    Assume on the contrary that there exists a good involution $\rho:M \to M$ on $(M,*)$.
    By the second condition of the definition of a good involution, we obtain 
    \begin{equation*}
        \langle x, \rho(y) \rangle \rho(y) = - \langle x,y \rangle y
    \end{equation*}
    for all $x,y \in M$. Thus we have 
    \begin{equation*}
        \langle e_{2}, \rho(e_{1}) \rangle \rho(e_{1})
        = - \langle e_{2}, e_{1} \rangle e_{1} = 3e_{1}
    \end{equation*}
    for $e_{1} = (1,0), e_{2} = (0,1) \in M$.
    Write $\rho(e_{1}) = (a,b)$. Since 
    \begin{equation*}
        \langle e_{2}, \rho(e_{1}) \rangle
        = -3a,
    \end{equation*}
    we have 
    \begin{equation*}
        \langle e_{2}, \rho(e_{1}) \rangle \rho(e_{1}) = (-3a^2, -3ab) = (6a^2,6ab).
    \end{equation*}
    Comparing the corresponding components, we obtain $6a^2 = 3$ and $6ab = 0$.
    Since no element satisfies $6a^2 = 3$, this is a contradiction.
    Thus there exists no good involution on $(M,*)$.
\end{proof}

\subsection*{Acknowledgements}
The authors are grateful to Takefumi Nosaka for helpful discussions and comments.
The authors would also like to thank the anonymous referee for helpful comments which improve the exposition.

\bibliographystyle{alpha}
\bibliography{symplectic_quandle}

@article {MR2493966,
    AUTHOR = {Navas, Esteban Adam and Nelson, Sam},
     TITLE = {On symplectic quandles},
   JOURNAL = {Osaka J. Math.},
  FJOURNAL = {Osaka Journal of Mathematics},
    VOLUME = {45},
      YEAR = {2008},
    NUMBER = {4},
     PAGES = {973--985},
      ISSN = {0030-6126},
   MRCLASS = {20N02 (57M25)},
  MRNUMBER = {2493966},
MRREVIEWER = {\'Agota\ Figula},
       URL = {http://projecteuclid.org.kyoto-u.idm.oclc.org/euclid.ojm/1227708829},
}

@article {MR4688855,
    AUTHOR = {Taniguchi, Yuta},
     TITLE = {Good involutions of generalized {A}lexander quandles},
   JOURNAL = {J. Knot Theory Ramifications},
  FJOURNAL = {Journal of Knot Theory and its Ramifications},
    VOLUME = {32},
      YEAR = {2023},
    NUMBER = {12},
     PAGES = {Paper No. 2350081, 7},
      ISSN = {0218-2165,1793-6527},
   MRCLASS = {20N02 (57K12)},
  MRNUMBER = {4688855},
MRREVIEWER = {Jonathan\ D. H. Smith},
       DOI = {10.1142/S0218216523500815},
       URL = {https://doi-org.kyoto-u.idm.oclc.org/10.1142/S0218216523500815},
}

@book {MR1878556,
    AUTHOR = {Lang, Serge},
     TITLE = {Algebra},
    SERIES = {Graduate Texts in Mathematics},
    VOLUME = {211},
   EDITION = {third},
 PUBLISHER = {Springer-Verlag, New York},
      YEAR = {2002},
     PAGES = {xvi+914},
      ISBN = {0-387-95385-X},
   MRCLASS = {00A05 (15-02)},
  MRNUMBER = {1878556},
       DOI = {10.1007/978-1-4613-0041-0},
       URL = {https://doi-org.kyoto-u.idm.oclc.org/10.1007/978-1-4613-0041-0},
}

@article {MR3001556,
    AUTHOR = {Albers, Peter and Frauenfelder, Urs},
     TITLE = {The space of linear anti-symplectic involutions is a
              homogenous space},
   JOURNAL = {Arch. Math. (Basel)},
  FJOURNAL = {Archiv der Mathematik},
    VOLUME = {99},
      YEAR = {2012},
    NUMBER = {6},
     PAGES = {531--536},
      ISSN = {0003-889X,1420-8938},
   MRCLASS = {53D35 (20G20 53D40)},
  MRNUMBER = {3001556},
       DOI = {10.1007/s00013-012-0461-4},
       URL = {https://doi-org.kyoto-u.idm.oclc.org/10.1007/s00013-012-0461-4},
}

@misc{ta2025goodinvolutionsconjugationsubquandles,
      title={Good involutions of conjugation subquandles}, 
      author={Luc Ta},
      year={2025},
      eprint={2505.08090},
      archivePrefix={arXiv},
      primaryClass={math.GT},
      url={https://arxiv.org/abs/2505.08090}, 
      note={\nolinkurl{arXiv:2505.08090}},
}

@misc{ta2025goodinvolutionstwistedconjugation,
      title={Good involutions of twisted conjugation subquandles and {A}lexander quandles}, 
      author={Luc Ta},
      year={2025},
      eprint={2508.16772},
      archivePrefix={arXiv},
      primaryClass={math.GT},
      url={https://arxiv.org/abs/2508.16772}, 
      note={\nolinkurl{arXiv:2508.16772}, to appear in \textit{Commun. Algebra}},
}

@article {MR21002,
    AUTHOR = {Takasaki, Mituhisa},
     TITLE = {Abstraction of symmetric transformations},
   JOURNAL = {T\^ohoku Math. J.},
  FJOURNAL = {The T\^ohoku Mathematical Journal},
    VOLUME = {49},
      YEAR = {1943},
     PAGES = {145--207},
      ISSN = {0040-8735,1881-2015},
   MRCLASS = {20.0X},
  MRNUMBER = {21002},
MRREVIEWER = {S.\ Kakutani},
}

@article {MR638121,
    AUTHOR = {Joyce, David},
     TITLE = {A classifying invariant of knots, the knot quandle},
   JOURNAL = {J. Pure Appl. Algebra},
  FJOURNAL = {Journal of Pure and Applied Algebra},
    VOLUME = {23},
      YEAR = {1982},
    NUMBER = {1},
     PAGES = {37--65},
      ISSN = {0022-4049,1873-1376},
   MRCLASS = {57M25 (20F29 20N05 53C35)},
  MRNUMBER = {638121},
MRREVIEWER = {Mark\ E.\ Kidwell},
       DOI = {10.1016/0022-4049(82)90077-9},
       URL = {https://doi-org.kyoto-u.idm.oclc.org/10.1016/0022-4049(82)90077-9},
}

@incollection {MR2371714,
    AUTHOR = {Kamada, Seiichi},
     TITLE = {Quandles with good involutions, their homologies and knot
              invariants},
 BOOKTITLE = {Intelligence of low dimensional topology 2006},
    SERIES = {Ser. Knots Everything},
    VOLUME = {40},
     PAGES = {101--108},
 PUBLISHER = {World Sci. Publ., Hackensack, NJ},
      YEAR = {2007},
      ISBN = {978-981-270-593-8; 981-270-593-7},
   MRCLASS = {57Q45},
  MRNUMBER = {2371714},
       DOI = {10.1142/9789812770967\_0013},
       URL = {https://doi-org.kyoto-u.idm.oclc.org/10.1142/9789812770967_0013},
}

@article {MR2657689,
    AUTHOR = {Kamada, Seiichi and Oshiro, Kanako},
     TITLE = {Homology groups of symmetric quandles and cocycle invariants
              of links and surface-links},
   JOURNAL = {Trans. Amer. Math. Soc.},
  FJOURNAL = {Transactions of the American Mathematical Society},
    VOLUME = {362},
      YEAR = {2010},
    NUMBER = {10},
     PAGES = {5501--5527},
      ISSN = {0002-9947,1088-6850},
   MRCLASS = {57M27 (57Q45)},
  MRNUMBER = {2657689},
MRREVIEWER = {Masahico\ Saito},
       DOI = {10.1090/S0002-9947-2010-05131-1},
       URL = {https://doi-org.kyoto-u.idm.oclc.org/10.1090/S0002-9947-2010-05131-1},
}

@article {MR672410,
    AUTHOR = {Matveev, S. V.},
     TITLE = {Distributive groupoids in knot theory},
   JOURNAL = {Mat. Sb. (N.S.)},
  FJOURNAL = {Matematicheski\u i\ Sbornik. Novaya Seriya},
    VOLUME = {119(161)},
      YEAR = {1982},
    NUMBER = {1},
     PAGES = {78--88, 160},
      ISSN = {0368-8666},
   MRCLASS = {57M25 (20L15)},
  MRNUMBER = {672410},
MRREVIEWER = {Jonathan\ A.\ Hillman},
}

@book {MR1853077,
    AUTHOR = {Cannas da Silva, Ana},
     TITLE = {Lectures on symplectic geometry},
    SERIES = {Lecture Notes in Mathematics},
    VOLUME = {1764},
 PUBLISHER = {Springer-Verlag, Berlin},
      YEAR = {2001},
     PAGES = {xii+217},
      ISBN = {3-540-42195-5},
   MRCLASS = {53Dxx (53-01)},
  MRNUMBER = {1853077},
MRREVIEWER = {Brendan\ J.\ Foreman},
       DOI = {10.1007/978-3-540-45330-7},
       URL = {https://doi-org.kyoto-u.idm.oclc.org/10.1007/978-3-540-45330-7},
}
\end{document}